\documentclass[preprint]{imsart}
\pdfoutput=1
\RequirePackage[OT1]{fontenc}
\usepackage{amsthm, amsmath, natbib, amssymb, bm, enumitem}
\RequirePackage[colorlinks,citecolor=blue,urlcolor=blue]{hyperref}
\usepackage{booktabs}

\startlocaldefs
\numberwithin{equation}{section}
\theoremstyle{plain}
\newtheorem{thm}{Theorem}[section]
\newtheorem{lemma}{Lemma}[section]
\theoremstyle{remark}
\newtheorem{remark}{Remark}[section]

\endlocaldefs

\def\sqa{\sqrt{a}}

\def\sql{\sqrt{\nu}}
\def\EE{\mathrm{E}}

\newcommand{\rd}{\mathrm{d}}
\newcommand{\JS}{\mathrm{JS}}
\newcommand{\mymid}{\,|\,}
\DeclareMathOperator*{\argmin}{arg\,min}
\allowdisplaybreaks

\begin{document}

\begin{frontmatter}
\title{Non-minimaxity of debiased shrinkage estimators}
\runtitle{Non-minimaxity of debiased estimators}

\begin{aug}
\author{\fnms{Yuzo} \snm{Maruyama}
\ead[label=e1]{maruyama@port.kobe-u.ac.jp}%
 \and
\fnms{Akimichi} \snm{Takemura}
\ead[label=e2]{a-takemura@biwako.shiga-u.ac.jp}
}

\runauthor{Y. Maruyama \& A. Takemura}

\address{Kobe University \& Shiga University\\
\printead{e1,e2}}

\end{aug}

\begin{abstract}
We consider the estimation of the $p$-variate normal mean of $X\sim N_p(\theta,I)$ 
under the quadratic loss function.
We investigate the decision theoretic properties of debiased shrinkage estimator,
the estimator which shrinks towards the origin for smaller $\|x\|^2$ and 
which is exactly equal to the unbiased estimator $X$ for larger $\|x\|^2$.
Such debiased shrinkage estimator seems superior to
the unbiased estimator $X$, which implies minimaxity.
However we show that it is not minimax under mild conditions.
\end{abstract}



\end{frontmatter}
\section{Introduction}
\label{sec:intro}
Let $ X$ have a $p$-variate normal distribution 
$ \mathcal{N}_{p} (\theta, I_{p}) $. 
We consider the problem of estimating the mean vector $\theta$ under 
the loss function 
\begin{align}
 L(\theta,\hat{\theta})=\| \hat{\theta} - \theta \|^2
=\sum_{i=1}^{p}(\hat{\theta}_i - \theta_{i})^2.
\end{align}
The risk function of an estimator $ \hat{\theta}(X)$ is
\begin{equation*}
 R(\theta,\hat{\theta})=\EE\bigl[ \| \hat{\theta}(X) - \theta \|^2\bigr]
=\int_{R^{p}}\frac{\| \hat{\theta}(x) - \theta \|^{2}}{(2\pi)^{p/2}}
\exp\left(-\frac{\| x - \theta \|^{2}}{2}\right)\rd x.  
\end{equation*}
The usual unbiased estimator $ X $ has the constant risk $p$ and is
minimax for $p\in\mathbb{N}$. 
\cite{Stein-1956} showed that there are orthogonally equivariant
estimators of the form 
\begin{align}\label{hat.theta.phi}
 \hat{\theta}_{\phi}(X)= \left( 1- \frac{\phi(\| X \|^2)}{\| X \|^2}\right)X
\end{align}
which dominate $X$ when $ p \geq 3$.
\cite{James-Stein-1961}
gave an explicit dominating procedure 
\begin{equation}\label{JS}
\hat{\theta}_{\JS}(X)=\left( 1- \frac{p-2}{\| X \|^2}\right)X,
\end{equation}  
called the James-Stein estimator. 
Further, as shown in \cite{Baranchik-1964}, the James-Stein estimator is inadmissible since
the positive-part estimator
\begin{equation}\label{JSPP}
\hat{\theta}_{\JS}^+(X)=\max\left(0, 1- \frac{p-2}{\| X \|^2}\right)X
\end{equation}  
dominates $\hat{\theta}_{\JS}$.
For a class of general shrinkage estimators $\hat{\theta}_{\phi}(X)$ given by \eqref{hat.theta.phi},
\cite{Baranchik-1970} proposed a sufficient condition for minimaxity, \{\ref{ba.1} and \ref{ba.2}\}
where
\begin{enumerate}[label= \textbf{B.\arabic*}, leftmargin=*]
\item\label{ba.1} $0\leq \phi(w)\leq 2(p-2)$ for all $w\geq 0$,
\item\label{ba.2} $\phi'(w)\geq 0$ for all $w\geq 0$.
\end{enumerate}
Further
\cite{Stein-1974} expressed the risk of
$\hat{\theta}_{\phi}(X)$ as
\begin{align}\label{stein.identity.1}
 \EE\bigl[\|\hat{\theta}_\phi-\theta\|^2\bigr]=
p+\EE\left[r_\phi(\|X\|^2)\right],
\end{align}
where 
\begin{align}\label{stein.identity.2}
 r_\phi(w)=\frac{\phi(w)}{w}\left\{\phi(w)-2(p-2)\right\}-4\phi'(w).
\end{align}
Hence the shrinkage factor $\phi(w)$ with the inequality $r_\phi(w)\leq 0$ for all $w\geq 0$,
implies minimaxity of $\hat{\theta}_\phi$. 
We see that \{\ref{ba.1} and \ref{ba.2}\} is a tractable sufficient condition 
for $r_\phi(w)\leq 0$ for all $w\geq 0$.

A series of papers, \cite{Efron-Morris-1971, Efron-Morris-1972-jasa, Efron-Morris-1972-biometrika, Efron-Morris-1973-jasa}, showed that
the James-Stein estimator can be interpreted as an empirical Bayes estimator under
$\theta\sim \mathcal{N}_p(0,\tau I_p)$.
Hence the shrinkage estimator including the James-Stein estimator utilize 
the prior information that $\|\theta\|^2$ is relatively small.
In fact, the risk function of the James-Stein estimator
is 
\begin{align}
 p-(p-2)^2\EE[1/\|X\|^2]
\end{align}
which is increasing in $\|\theta\|^2$.
On the other hand, 
the larger $\|x\|^2$ suggests that the prior information ($\|\theta\|^2$ is relatively small) 
is incorrect. 
Although the James-Stein estimator uniformly dominates $X$ under the quadratic risk,
for larger $\|x\|^2$, 
the unbiased estimator $X$ seems superior to 
the shrinkage estimators with the bias given by
\begin{align}
 \EE\left[\left( 1- \frac{\phi(\| X \|^2)}{\| X \|^2}\right)X\right]-\theta
=-\EE\left[\frac{\phi(\| X \|^2)}{\| X \|^2}X\right],
\end{align}
with $O(1/\|\theta\|)$ provided $\phi(w)$ is bounded.
Note that many popular shrinkage estimators have $\phi$ with 
\begin{align}
 \liminf_{w\to\infty} \phi(w)\geq p-2.
\end{align}
See a sufficient condition for admissibility by \cite{Brown-1971}.

In this paper, we define debiased shrinkage estimator by 
\begin{enumerate}[label= \textbf{DS.\arabic*},  leftmargin=*]
\item \label{as.ds.2} $\phi(w)$ is weakly differentiable with bounded $\phi'(w)$,
\item \label{as.ds.3} For some $a>0$, $0<\phi(w)\leq w$ on $(0,a)$ and $ \phi(w)=0$ on $[a,\infty)$.
\end{enumerate}
Hence the debiased shrinkage estimator shrinks 
towards the origin for smaller $\|x\|^2$ and 
is exactly equal to the unbiased estimator $X$ for larger $\|x\|^2$.
Such debiased shrinkage estimators seem superior to
the unbiased estimator $X$, which implies minimaxity.
In this paper, we are interested in whether 
the debiased shrinkage estimators are minimax or not.

In the literature, there are some debiased estimators including
SCAD (Smoothly Clipped Absolute Deviation) by \cite{Fan-Li-2001} and
nearly unbiased estimators by MCP (Minimax Concave Penalty) by \cite{Zhang-2010},
which have not necessarily aimed at enjoying the conventional minimaxity.

The organization of this paper is as follows.
By \eqref{stein.identity.1}, 
the risk difference between $ \hat{\theta}$ and the minimax estimator $X$
is given by
\begin{align}\label{Del.nu}
 \EE\bigl[\|\hat{\theta}_\phi-\theta\|^2\bigr]-p=\EE[r_\phi(\|X\|^2)]
=
\EE\left[r_\phi(\|X\|^2)I_{[0,a]}(\|X\|^2)\right],
\end{align}
where $r_\phi(w)$ is given by \eqref{stein.identity.2} and
the second equality follows from \ref{as.ds.3}.
In Section \ref{sec:as}, we give a useful result, Theorem \ref{theorem:non_central_chi_1}, 
on the asymptotic behavior of this type of
an expected value when $\|\theta\|^2\to\infty$.
In Section \ref{sec:scad.mcp},
we review SCAD and MCP as a solution of penalized least squares and
investigate how the corresponding $\phi(w)$ approaches $0$ as $w\nearrow a$.
In Section \ref{sec:minimax}, using Theorem \ref{theorem:non_central_chi_1}, 
we show that the debiased shrinkage estimators with \ref{as.ds.2} and \ref{as.ds.3} 
as well as mild conditions on the way how $\phi(w)$ approaches $0$ as $w\nearrow a$, 
are not minimax, which is not necessarily expected. 

\section{Asymptotic behavior of an expected value}
\label{sec:as}
For fixed $a>0$, we investigate the asymptotic behavior of the expected value 
\begin{align}
G(\|\theta\|^2;a)=\EE\left[g(\|X\|^2)I_{[0,a]}(\|X\|^2)\right] \text{ as }\|\theta\|^2\to\infty
\end{align}
where $g(w)$ satisfies \ref{as.g.1} and \ref{as.g.3}:
\begin{enumerate}[label= \textbf{A.\arabic*},  leftmargin=*]
\item \label{as.g.1} $w^{(p-1)/2}|g(w)|$ is bounded on $[0,a]$.
\item \label{as.g.3} There exists a nonnegative real $b$ such that
\begin{align}\label{eq:left_cont}
 \lim_{w\nearrow a}\frac{g(w)}{(a-w)^{b}}=1.
\end{align}
\end{enumerate}
Notice that, on \ref{as.g.3}, we do not lose the generality even if 
we assume the limit of $g(w)/(a-w)^b$ is $1$. If the limit is equal to $g_*(\neq 0)$,
we have only to consider 
\begin{align*}
g_* \EE\left[\frac{g(\|X\|^2)}{g_*}I_{[0,a]}(\|X\|^2)\right].
\end{align*}
Then we have the following result.
\begin{thm}\label{theorem:non_central_chi_1}
Assume $p\geq 2$ and that $g(w)$ satisfies \ref{as.g.1} and \ref{as.g.3}.
Let $\nu=\|\theta\|^2$ and
\begin{equation}\label{eq:c}
c(a,b,p)= \frac{a^{(p-1)/4+b/2}2^b\Gamma(b+1)}{\sqrt{2\pi}\exp(a/2)}.
\end{equation}
Then 
\begin{equation*}
 \lim_{\nu\to\infty}\frac{\nu^{(p+1)/4+b/2}e^{\nu/2}}{e^{\sql\sqa}} G(\nu;a)
=c(a,b,p).
\end{equation*}
\end{thm}

\begin{proof}
We first prove the theorem under the proper subset of \ref{as.g.1};
\begin{enumerate}[label= \textbf{A.1.\arabic*},  leftmargin=45pt]
\item \label{as.g.1.1} $|g(w)|$ is bounded on $[0,a]$.
\end{enumerate}
Note that $\|X\|^2$ can be decomposed as $U^2+V$ where $U\sim N(\sqrt{\nu},1)$, $V\sim \chi^2_{p-1}$
and $U$ and $V$ are mutually independent.
Then we have
\begin{align*}
G(\nu;a) &=\EE\left[g(U^2+V)I_{[0,a]}(U^2+V)\right] \\
&=\EE_U\left[\EE_V\left[g(U^2+V)I_{[0,a]}(U^2+V)\mymid U=u\right]I_{[-\sqrt{a},\sqrt{a}]}(U)\right] \\
&=\EE_U\left[\EE_V\left[g(V+a-\{a-U^2\})I_{[0,a-U^2]}(V)\mymid U=u\right]I_{[-\sqrt{a},\sqrt{a}]}(U)\right] \\
&=\EE_U\left[H(a-U^2)(a-U^2)^{b+q}I_{[-\sqrt{a},\sqrt{a}]}(U)\right] ,
\end{align*}
where $q=(p-1)/2$, $H(\cdot)$ is given by 
\begin{align}\label{eq:G}
H(y)=\frac{1}{y^{b+q}}\int_0^y g(v+a-y)f_{p-1}(v)\rd v
\end{align}
and $f_{p-1}(v)$ is the pdf of $\chi^2_{p-1}$.
Hence $G(\nu;a)$ is rewritten as
\begin{align*}
G(\nu;a)= \int_{-\sqa}^{\sqa} H(a-u^2) 
(a-u^2)^{b+q}\frac{1}{\sqrt{2\pi}} e^{-(u-\sql)^2/2}\rd u.
\end{align*}
Since the asymptotic behavior of $G(\nu;a)$ as $\nu\to\infty$ is of interest, 
$\nu>a$ is assumed in the following.
For $G(\nu;a)$, apply the change of variables,
\begin{equation*}
 z=(\sql-\sqa)(-u+\sqa)
\end{equation*}
which implies
\begin{equation*}
 u=\sqa -\frac{z}{\sql-\sqa}, \ \rd u=-\frac{1}{\sql-\sqa}\rd z.
\end{equation*}
Then we have
\begin{align*}
&G(\nu;a)\\
&=\frac{1}{\sqrt{2\pi}(\sql-\sqa)} \int_0^{2\sqa(\sql-\sqa)}
\left(a-\left\{\sqa -\frac{z}{\sql-\sqa}\right\}^2\right)^{b+q}
\notag \\
&\quad \times
H\left(a-\left\{\sqa -\frac{z}{\sql-\sqa}\right\}^2\right)
\exp\left(-\frac{1}{2} \left\{\sqa -\frac{z}{\sql-\sqa}-\sql\right\}^2\right)\rd z.\notag
\end{align*}
Further we rewrite it as
\begin{align*}
& G(\nu;a)\\
&=\frac{\exp(-\{\sql-\sqa\}^2/2)}
{\sqrt{2\pi}(\sql-\sqa)} \int_0^{2\sqa(\sql-\sqa)}
H\left(a-\left\{\sqa -\frac{z}{\sql-\sqa}\right\}^2\right) \\
&\qquad \times\left(\frac{2\sqa z}{\sql-\sqa}-\frac{z^2}{(\sql-\sqa)^2}\right)^{b+q} 
 \exp\left(-z-\frac{z^2}{2(\sql-\sqa)^2}\right)\rd z \\
&=\frac{\exp(-\{\sql-\sqa\}^2/2)2^{b+q}a^{(b+q)/2}}{\sqrt{2\pi}(\sql-\sqa)^{b+q+1}}
\int_0^\infty z^{b+q}e^{-z}H_1(z;\nu)\rd z,
\end{align*}
where 
\begin{equation}\label{H1H1H1}
 \begin{split}
H_1(z;\nu)&=H\left(a-\left\{\sqa -\frac{z}{\sql-\sqa}\right\}^2\right) 
\left(1-\frac{z}{2\sqa (\sql-\sqa)}\right)^{b+q} \\
&\qquad \times \exp\left(-\frac{z^2}{2(\sql-\sqa)^2}\right)I_{(0,2\sqa(\sql-\sqa)) }(z).
\end{split}
\end{equation}
From Part \ref{lem:G*_2} of Lemma \ref{lem:G*} below, $H(y)$ on $[0,a]$ is bounded under \ref{as.g.1.1}.
Hence, for any $\nu$, we have
\begin{align}\label{bound.H1}
 H_1(z;\nu)
\begin{cases}
\displaystyle\leq \max_{y\in[0,a]}|H(y)| & 0\leq z\leq 2\sqa (\sql-\sqa)\\
=0 & z>2\sqa (\sql-\sqa).
\end{cases}
\end{align}
Further, by \eqref{H1H1H1} and Part \ref{lem:G*_1} of Lemma \ref{lem:G*}, we have
\begin{align}\label{limit.H1}
 \lim_{\nu\to\infty}H_{1}(z;\nu)=\lim_{y\to 0}H(y)=\frac{\Gamma(b+1)}{\Gamma(b+(p+1)/2)2^{(p-1)/2}}.
\end{align}
By \eqref{bound.H1} and \eqref{limit.H1}, the dominated convergence theorem, 
gives
\begin{align*}
&\lim_{\nu\to\infty} \int_0^\infty z^{b+q}e^{-z}H_1(z;\nu)\rd z
= \int_0^\infty z^{b+q}e^{-z}\lim_{\nu\to\infty}H_1(z;\nu)\rd z 
= \frac{\Gamma(b+1)}{2^{(p-1)/2}},
\end{align*}
which completes the proof under \ref{as.g.1.1}.

\medskip

Now we assume \ref{as.g.1}, that is, $w^{(p-1)/2}|g(w)|$ is bounded on $[0,a]$ as
\begin{align}
 w^{(p-1)/2}|g(w)|<M.
\end{align}
Let $f_p(w,\nu)$ be the density of $W=\|X\|^2$.
Note that $ f_p(w,\nu)/f_p(w,0)$ for any fixed $\nu>0$ is increasing in $w$ and that
$ w^{-(p-1)/2}$ is decreasing in $w$.
By the correlation inequality, we have
\begin{align*}
\left| \int_0^{a/2} g(w) f_p(w,\nu)\rd w \right| 
&< \int_0^{a/2} |g(w)| f_p(w,\nu)\rd w  \\
&<M \int_0^{a/2} w^{-(p-1)/2} f_p(w,\nu)\rd w \\
&=M \int_0^{a/2} w^{-(p-1)/2} \frac{f_p(w,\nu)}{f_p(w,0)}f_p(w,0)\rd w \\
& \leq M \frac{\int_0^{a/2} w^{-(p-1)/2}f_p(w,0)\rd w}
{\int_0^{a/2} f_p(w,0)\rd w}\int_0^{a/2} f_p(w,\nu)\rd w \\
& \leq M_1 \int_0^{a/2} f_p(w,\nu)\rd w
\end{align*}    
where 
\begin{equation*}
 M_1=M \frac{\int_0^{a/2} w^{-(p-1)/2}f_p(w,0)\rd w}
{\int_0^{a/2} f_p(w,0)\rd w}.
\end{equation*}
Let 
\begin{align*}
 g_L(w)=
\begin{cases}
 -M_1 & 0\leq w\leq a/2 \\
g(w) & a/2<w\leq a,
\end{cases} \quad
 g_U(w)=
\begin{cases}
 M_1 & 0\leq w\leq a/2 \\
g(w) & a/2<w\leq a,
\end{cases}
\end{align*}
which are both bounded. 
Then we have
\begin{align*}
\EE\left[g_L(W)I_{[0,a]}(W)\right]<G(\nu;a)<\EE\left[g_U(W)I_{[0,a]}(W)\right]
\end{align*}
and, by the result under \ref{as.g.1.1},
 \begin{align*}
& \lim_{\nu\to\infty}\frac{\nu^{(p+1)/4+b/2}e^{\nu/2}}{e^{\sql\sqa}} 
\EE\left[g_L(W)I_{[0,a]}(W)\right]\\
&=
 \lim_{\nu\to\infty}\frac{\nu^{(p+1)/4+b/2}e^{\nu/2}}{e^{\sql\sqa}} 
\EE\left[g_U(W)I_{[0,a]}(W)\right]
=c(a,b,p),
\end{align*}
where $c(a,b,p)$ is given by \eqref{eq:c}.
Hence Theorem \ref{theorem:non_central_chi_1} is valid for the case where
$w^{(p-1)/2}|g(w)|$ is bounded.
\end{proof}

The following lemma gives some properties on $H(y)$ given by \eqref{eq:G},
needed in the proof of Theorem \ref{theorem:non_central_chi_1}.
\begin{lemma}\label{lem:G*}
We assume that $ |g(w)|$ is bounded on $[0,a]$ as in \ref{as.g.1.1}.
Then we have the following results.
\begin{enumerate}
 \item \label{lem:G*_1}
$\displaystyle \lim_{y\to 0}H(y)=\frac{\Gamma(b+1)}{\Gamma(b+(p+1)/2)2^{(p-1)/2}}$.
\item \label{lem:G*_2}
$ H(y)$ on $[0,a]$ is bounded.
\end{enumerate}
 \end{lemma}
\begin{proof}
By \eqref{eq:left_cont}, for any $\epsilon>0$, there exists $\delta_1(\epsilon)>0$ such that
\begin{equation*}
 (1-\epsilon)(a-w)^b \leq g(w)\leq (1+\epsilon)(a-w)^b
\end{equation*}
for all $a-\delta_1\leq w < a$ and hence
\begin{align}
 (1-\epsilon)(y-v)^b \leq g(v+a-y)\leq (1+\epsilon)(y-v)^b
\end{align}
for all $0<v<y\leq \delta_1(\epsilon)$.
Further, for any $0<\epsilon<1$, we have 
$1-\epsilon \leq  e^{-v/2} $ for all $0<v\leq\delta_2(\epsilon)$ where
$\delta_2(\epsilon)=-2\log(1-\epsilon)$ and hence
\begin{equation*}
 (1-\epsilon)\frac{v^{q-1}}{\Gamma(q)2^{q}} \leq f_{p-1}(v) \leq \frac{v^{q-1}}{\Gamma(q)2^{q}} 
\end{equation*}
with $q=(p-1)/2$ and for all $v\leq \delta_2(\epsilon)$. 
Then for any $\epsilon>0$ and all $0<y\leq \min(\delta_1(\epsilon),\delta_2(\epsilon))$,
we have
\begin{align*}
 (1-\epsilon)^2
\int_0^y \frac{(y-v)^b}{y^{b+q}}\frac{v^{q-1}}{\Gamma(q)2^{q}}\rd v\leq H(y) \leq
(1+\epsilon)\int_0^y \frac{(y-v)^b}{y^{b+q}}\frac{v^{q-1}}{\Gamma(q)2^{q}}\rd v,
\end{align*}
where
\begin{align}
 \int_0^y \frac{(y-v)^b}{y^{b+q}}v^{q-1}\rd v=B(q,b+1)=B((p-1)/2,b+1).
\end{align}
Hence, for $0<y\leq \min(\delta_1(\epsilon),\delta_2(\epsilon))$, we have
\begin{align*}
 \frac{(1-\epsilon)^2\Gamma(b+1)}{\Gamma(b+(p+1)/2)2^{(p-1)/2}}
\leq H(y)\leq \frac{(1+\epsilon)\Gamma(b+1)}{\Gamma(b+(p+1)/2)2^{(p-1)/2}}.
\end{align*}
and the part \ref{lem:G*_1} follows.

By \eqref{eq:G}, we have
\begin{align}
 H(a)=\frac{1}{a^{b+(p-1)/2}}\int_0^a g(v)f_{p-1}(v)\rd v,
\end{align}
which is bounded under \ref{as.g.1.1}.
By the continuity of $H(y)$ and Part \ref{lem:G*_1} of this lemma, 
the part \ref{lem:G*_2} follows.
\end{proof}

\section{Review of existing debiased shrinkage estimators}
\label{sec:scad.mcp}
As we mentioned in Section \ref{sec:intro},
in the literature, there are some ``debiased shrinkage'' estimators including
SCAD (Smoothly Clipped Absolute Deviation) by \cite{Fan-Li-2001} and
nearly unbiased estimators by MCP (Minimax Concave Penalty) by \cite{Zhang-2010}, although
they do not necessarily aim at enjoying the conventional minimaxity.
In this section, we assume $p=1$ and 
review existing estimators as solutions of the penalized least squares problem;
\begin{align}\label{hat.theta}
 \hat{\theta}(P;\lambda)=\argmin_\theta\left\{ (\theta-x)^2+P(|\theta|;\lambda)\right\}.
\end{align}
Table \ref{table:rsh} summarizes three popular penalty functions $P(|\theta|;\lambda)$,
and the corresponding minimizers ``ridge'', ``soft thresholding'' and ``hard thresholding''.
\begin{table}
 \setlength{\tabcolsep}{2pt}
\caption{Ridge, Soft-thresholding, Hard-thresholding}
\label{table:rsh}
\begin{tabular}{cccc} \toprule
 & $P(|\theta|;\lambda)$ & $\hat{\theta}(P;\lambda)$ & $\phi$ \\ \midrule
ridge &  $\displaystyle\lambda \theta^2$ 
&  $\displaystyle \hat{\theta}_{\mathrm{R}}=(1-1/\{\lambda+1\})x$ &
\ref{as.ds.2}  
\\
\begin{tabular}{c}
soft \\ thresholding 
\end{tabular} &  $\displaystyle\lambda |\theta|$ & 
$\displaystyle \hat{\theta}_{\mathrm{ST}}=\begin{cases}
		0 & x^2\leq \lambda^2 \\ (1-\lambda/|x|)x & x^2> \lambda^2
	       \end{cases}$ &
\ref{as.ds.2}  \\
\begin{tabular}{c}
hard \\ thresholding 
\end{tabular}  &  
$\displaystyle \lambda^2-(|\theta|-\lambda)^2I_{[0,\lambda]}(|\theta|)$ 
& 
$\displaystyle \hat{\theta}_{\mathrm{HT}}=\begin{cases}
		0 & x^2\leq \lambda^2 \\ x & x^2> \lambda^2
	       \end{cases}$&
\ref{as.ds.3}  
 \\ \bottomrule
\end{tabular}
\end{table}
For the three estimators, 
the corresponding shrinkage factors, $\phi(x^2)$,
from the form
\begin{align}
 \hat{\theta}=\left(1-\frac{\phi(x^2)}{x^2}\right)x
\end{align}
are
\begin{align*}
 \phi_{\mathrm{R}}(w)=\frac{w}{\lambda+1},\quad 
 \phi_{\mathrm{ST}}(w)=\begin{cases}
			w & w \leq \lambda^2 \\ \lambda w^{1/2} & w>\lambda^2
		       \end{cases},\quad 
 \phi_{\mathrm{HT}}(w)=\begin{cases}
			w & w \leq \lambda^2 \\ 0 & w>\lambda^2
		       \end{cases}.
\end{align*}
We see that
\ref{as.ds.3}, \ref{as.ds.3} and \ref{as.ds.2} are not satisfied by 
$\phi_{\mathrm{R}}(w)$, $ \phi_{\mathrm{ST}}(w)$ and $ \phi_{\mathrm{HT}}(w)$, respectively.

SCAD (Smoothly Clipped Absolute Deviation) by \cite{Fan-Li-2001} is the minimizer,
\eqref{hat.theta}, with the continuous differentiable penalty function defined by
\begin{align}
 P'(|\theta|;\lambda,\alpha)
=\begin{cases}
\lambda & |\theta|<\lambda \\
\displaystyle\frac{\alpha\lambda - |\theta|}{\alpha-1} & \lambda\leq |\theta|<\alpha\lambda \\
0 & |\theta|\geq \alpha\lambda
 \end{cases}
\end{align}
where $\alpha>2$. 
The resulting solution is 
\begin{align}
\hat{\theta}_{\mathrm{SCAD}}(\lambda;\alpha)
=
\begin{cases}
 0 & 0<x^2< \lambda^2 \\
(1-\lambda/|x|)x & \lambda^2 \leq x^2 \leq 4\lambda^2 \\
\displaystyle \left(1-\frac{-x^2+\alpha\lambda |x|}{(\alpha-2)x^2}\right)x & 4\lambda^2 \leq x^2 \leq \alpha^2 \lambda^2\\
x & x^2\geq \alpha^2\lambda^2,
\end{cases} 
\end{align}
where the corresponding shrinkage factor is
\begin{align}\label{phi.scad}
\phi_{\mathrm{SCAD}}(w)
=
\begin{cases}
 w & 0<w< \lambda^2 \\
\lambda w^{1/2} &\lambda^2 \leq w \leq 4\lambda^2 \\
\displaystyle \frac{-w+\alpha\lambda w^{1/2}}{a-2} & 4\lambda^2 \leq w \leq \alpha^2 \lambda^2\\
0 & w\geq \alpha^2\lambda^2.
\end{cases} 
\end{align}
We see that $\phi_{\mathrm{SCAD}}(w)$ satisfies both \ref{as.ds.2} and \ref{as.ds.3}. 
Further, by \eqref{phi.scad}, the derivative at $w=\alpha^2\lambda^2$ is
\begin{align}\label{eq:deriv.scad}
 \frac{\rd}{\rd w}\phi_{\mathrm{SCAD}}(w)|_{w=\alpha^2\lambda^2}=-\frac{1}{2(\alpha-2)}<0.
\end{align}
As pointed in \cite{Strawderman-Wells-2012},
the nearly unbiased estimator by MCP (Minimax Concave Penalty) considered in
\cite{Zhang-2010} is equivalent to 
the minimizer of \eqref{hat.theta} with the continuous differentiable penalty function defined by
\begin{align}
 P'(|\theta|;\lambda,\alpha)
=\begin{cases}
2\{\lambda-|\theta|/\alpha\} & |\theta|<\alpha\lambda \\
0 & |\theta|\geq \alpha\lambda
 \end{cases}
\end{align}
where $\alpha>1$. 
Then the resulting solution is given by
\begin{align}
\hat{\theta}_{\mathrm{MCP}}(\lambda;\alpha)
=
\begin{cases}
 0 & 0<x^2< \lambda^2 \\
\displaystyle \left(1-\frac{-x^2+\alpha\lambda |x|}{(\alpha-1)x^2}\right)x & \lambda^2 \leq x^2 \leq \alpha^2 \lambda^2\\
x & x^2\geq \alpha^2\lambda^2,
\end{cases} 
\end{align}
where the corresponding shrinkage factor is
\begin{align}\label{phi.mcp}
\phi_{\mathrm{MCP}}(w)
=
\begin{cases}
 w & 0<w< \lambda^2 \\
\displaystyle \frac{-w+\alpha\lambda w^{1/2}}{\alpha-1} & \lambda^2 \leq w \leq \alpha^2 \lambda^2\\
0 & w\geq \alpha^2\lambda^2.
\end{cases} 
\end{align}
We see that $\phi_{\mathrm{MCP}}(w)$ satisfies both \ref{as.ds.2} and \ref{as.ds.3}. 
Further, by \eqref{phi.mcp}, the derivative at $w=\alpha^2\lambda^2$ is
\begin{align}\label{eq:deriv.mcp}
 \frac{\rd}{\rd w}\phi_{\mathrm{MCP}}(w)|_{w=\alpha^2\lambda^2}=-\frac{1}{2(\alpha-1)}<0.
\end{align}
By \eqref{eq:deriv.scad} and \eqref{eq:deriv.mcp},
both $\phi_{\mathrm{SCAD}}(w)$ and $\phi_{\mathrm{MCP}}(w)$ 
approach $0$ as $w\nearrow \alpha^2\lambda^2$ with the negative slope.

Aside from the justification as a solution of the penalized least squares problem \eqref{hat.theta},
let us consider
\begin{align}
 \phi_{\mathrm{Q}}(w)=
\begin{cases}
 w & \displaystyle 0\leq w<\frac{2a+1-\sqrt{4a+1}}{2} \\
(a-w)^2 & \displaystyle  \frac{2a+1-\sqrt{4a+1}}{2} \leq w \leq a \\
0 & w>a.
\end{cases}
\end{align}
We see that $ \phi_{\mathrm{Q}}(w)$ satisfies
both \ref{as.ds.2} and \ref{as.ds.3} and
\begin{align}\label{eq:deriv.Q}
\lim_{w\to a} \frac{\rd}{\rd w} \phi_{\mathrm{Q}}(w)=0
\quad\text{and}\quad
\lim_{w\nearrow a} (a-w)\frac{(\rd/\rd w) \phi_{\mathrm{Q}}(w)}{\phi_{\mathrm{Q}}(w)}=-2.
\end{align}

When $\phi(w)$ of debiased shrinkage estimator
approaches $0$ from above as $w\to a$,
it seems that both \{\eqref{eq:deriv.scad} and \eqref{eq:deriv.mcp}\}
and \eqref{eq:deriv.Q} are typical behaviors characterized by $\phi'(w)$.


\section{Main result}
\label{sec:minimax}
In this section, we investigate the minimaxity of the shrinkage debiased estimators
with \ref{as.ds.2} and \ref{as.ds.3}. 
Recall, as in \eqref{Del.nu},
the risk difference between $ \hat{\theta}$ and the minimax estimator $X$
is 
\begin{align}\label{Del.nu.1}
 \EE\bigl[\|\hat{\theta}_\phi-\theta\|^2\bigr]-p
=\EE\left[r_\phi(\|X\|^2)I_{[0,a]}(\|X\|^2)\right],
\end{align}
where $r_\phi(w)$ is given by \eqref{stein.identity.2}.
Under the assumptions on $\phi(w)$, \ref{as.ds.2} and \ref{as.ds.3},
$r_\phi(w)$ given by \eqref{stein.identity.2} is bounded, that is, there exists an $M$
such that
\begin{align}\label{r.bounded}
 |r_\phi(w)|<M \text{ on }[0,a].
\end{align}
For $\phi(w)$ with $\lim_{w\nearrow a}\phi(w)=0$ as wells as
$\phi(w)>0$ for $w<a$, we consider two cases
as a generalization of \{\eqref{eq:deriv.scad} and \eqref{eq:deriv.mcp}\}
and \eqref{eq:deriv.Q}:
\begin{enumerate}[label= \textbf{Case \arabic*},  leftmargin=*]
\item\label{ca.1} $\limsup_{w\nearrow a}\phi'(w)<0$.
\item\label{ca.2} $\lim_{w\nearrow a}\phi'(w)=0$ and there exist $0<\epsilon<1$  and $\gamma>1$ such that
\begin{align}\label{eq.ca.2}
-\gamma < (a-w)\frac{\phi'(w)}{\phi(w)}< -\frac{1}{\gamma},
\end{align}
for all $w\in(\epsilon a,a)$.
\end{enumerate}

Under \ref{ca.1}, there exist $\delta_1>0$ and $0<\delta_2<1$ such that
\begin{align}\label{r.phi.ca.1}
 \phi'(w) < - \delta_1 \ \text{ and } \ \frac{\phi(w)}{w}\left\{\phi(w)-2(p-2)\right\} > - \delta_1,
\end{align}
for all $w\in(\delta_2 a, a)$.
Then, by \eqref{stein.identity.2} and \eqref{r.phi.ca.1}, we have
\begin{equation}\label{r.phi.geq.1}
  r_\phi(w) =\frac{\phi(w)}{w}\left\{\phi(w)-2(p-2)\right\}-4\phi'(w) 
\geq  3\delta_1>0,
\end{equation}
for all $w\in(\delta_2 a, a)$.
Hence, by Theorem \ref{theorem:non_central_chi_1} with \eqref{Del.nu.1}, \eqref{r.bounded} and \eqref{r.phi.geq.1},
we have
\begin{equation}
 \begin{split}
& \liminf_{\nu\to\infty}\frac{\nu^{(p+1)/4}e^{\nu/2}}{e^{\sql\sqa}} 
 \left\{\EE\bigl[\|\hat{\theta}_\phi-\theta\|^2\bigr]-p\right\}\\
&\quad \geq 3\delta_1c(a,0,p) >0,
\end{split}
\end{equation}
which implies that the debiased shrinkage estimator is not minimax under \ref{ca.1}.

Under \ref{ca.2}, the inequality
\begin{equation*}
-\gamma < (a-w)\frac{\phi'(w)}{\phi(w)} 
\end{equation*}
for $w\in(\epsilon a,a)$ implies
\begin{equation*}
\int_{\epsilon a}^w  \frac{\phi'(t)}{\phi(t)} \rd t>\int_{\epsilon a}^w\frac{-\gamma}{a-t}\rd t
\end{equation*}
which is equivalent to
\begin{align}\label{phi.lower}
 \phi(w)>\phi_*(a-w)^\gamma,\ \text{ where } \
\phi_*=\frac{\phi(\epsilon a)}{(a-\epsilon a)^\gamma},
\end{align}
for all $w\in(\epsilon a, a)$.
Further let
\begin{align}
 \epsilon'=\frac{1}{1+1/\{(p-2)\gamma\}}.
\end{align}
Then, for $w\in(\epsilon'a,a)$,
we have
\begin{align}\label{ca.ca.ca.2}
 -\frac{2(p-2)(a-w)}{w}\geq - \frac{2}{\gamma}.
\end{align}
Hence, for $w\in(\max(\epsilon,\epsilon')a,a)$, we have
\begin{equation}\label{r.phi.geq.2}
 \begin{split}
 r_\phi(w)
&=\frac{\{\phi(w)\}^2}{w}-2(p-2)\frac{\phi(w)}{w}-4\phi'(w) \\
&\geq -2(p-2)\frac{\phi(w)}{w}-4\phi'(w) \\
&= \frac{\phi(w)}{a-w}\left\{ -\frac{2(p-2)(a-w)}{w}-4(a-w)\frac{\phi'(w)}{\phi(w)}\right\} \\
&\geq \frac{2}{\gamma}\frac{\phi(w)}{a-w},
\end{split}
\end{equation}
where the second inequality follows from \eqref{eq.ca.2} and \eqref{ca.ca.ca.2}.
Further, by \eqref{phi.lower} and \eqref{r.phi.geq.2}, we have
\begin{align}\label{r.phi.geq.3}
 r_\phi(w)\geq \frac{2\phi_*}{\gamma}(a-w)^{\gamma-1}
\end{align}
for $w\in(\max(\epsilon,\epsilon')a,a)$.
Hence, by Theorem \ref{theorem:non_central_chi_1} with \eqref{Del.nu.1}, \eqref{r.bounded} and \eqref{r.phi.geq.3}, we have
\begin{equation*}
 \begin{split}
&\liminf_{\nu\to\infty}\frac{\nu^{(p+1)/4+(\gamma-1)/2}e^{\nu/2}}{e^{\sql\sqa}} 
\left\{\EE\bigl[\|\hat{\theta}_\phi-\theta\|^2\bigr]-p\right\}\\
&\quad\geq\frac{2\phi_*}{\gamma}c(a,\gamma-1,p) >0,
\end{split}
\end{equation*}
which implies that the debiased shrinkage estimator not minimax under \ref{ca.2}.
In summary, we have the following theorem.
\begin{thm}
The debiased shrinkage estimator with \ref{as.ds.2} and \ref{as.ds.3} is not minimax
under either \ref{ca.1} or \ref{ca.2}.
\end{thm}

\begin{remark}
Yet another application of Theorem \ref{theorem:non_central_chi_1} is also related to Stein estimation,
the gain of the positive-part estimator $\hat{\theta}_{\JS}^+$ given by \eqref{JSPP}
over
the naive James-Stein estimator $\hat{\theta}_{\JS}$ given by \eqref{JS}.
For these estimators, the corresponding $\phi(w)$ are given by
\begin{equation}\label{JS.phi}
 \phi_{\JS}^+(w)=\min(w,p-2),\quad\phi_{\JS}(w)=p-2.
\end{equation}
By the general expression of the risk, 
\eqref{stein.identity.1} and \eqref{stein.identity.2} with \eqref{JS.phi},
we have
\begin{equation}\label{js.jspp.risk.diff}
 R(\theta,\hat{\theta}_{\JS})-R(\theta,\hat{\theta}_{\JS}^+)
=\EE\Bigl[\bigl\{-\frac{(p-2)^2}{\|X\|^2} + 2p - \|X\|^2\bigr\}I_{[0,p-2]}(\|X\|^2)\Bigr].
\end{equation}

Let 
\begin{align*}
f_k(v)&=\frac{v^{k/2-1}e^{-v/2}}{\Gamma(k/2)2^{k/2}}, \quad
f_k(v;\nu)=\sum_{i=0}^\infty\frac{(\nu/2)^i\exp(-\nu/2)}{i!}f_{k+2i}(v),\\
F_k(v)&=\int_0^v f_k(w)\rd w, \quad F_k(v;\nu)=\int_0^v f_k(w;\nu)\rd w.
\end{align*}
\cite{Hansen-2022b}, in Theorem 15.7, 
expressed the risk difference \eqref{js.jspp.risk.diff} through 
$F_k(v)$ and $F_k(v;\nu)$, 
the distribution functions of the central chi-square with and non-central chi-square, as
\begin{align*}
 R(\theta,\hat{\theta}_{\JS})-R(\theta,\hat{\theta}_{\JS}^+)
&=2pF_p(p-2;\nu)-pF_{p+2}(p-2;\nu)-\nu F_{p+4}(p-2;\nu)\\
&\quad -(p-2)^2\sum_{i=0}^\infty\frac{(\nu/2)^i\exp(-\nu/2)}{i!}\frac{F_{p+2i-2}(p-2)}{p+2i-2}.
\end{align*}
\cite{Robert-1988} expressed the risk difference \eqref{js.jspp.risk.diff} through 
the Dawson integral given by
\begin{align*}
 D(\lambda)=e^{-\lambda^2}\int_0^\lambda e^{t^2}\rd t.
\end{align*}
The results by \cite{Hansen-2022b} and \cite{Robert-1988}, 
do not seem to directly 
provide the exact asymptotic order of the major term of \eqref{js.jspp.risk.diff}
with the exact coefficient.

Using Theorem \ref{theorem:non_central_chi_1}, we can get it as follows.
Since
\begin{align*}
\left. \Bigl\{-\frac{(p-2)^2}{w} + 2p - w\Bigr\}\right|_{w=p-2}=4,
\end{align*}
and
\begin{align*}
 w^{(p-1)/2}\left|-\frac{(p-2)^2}{w} + 2p - w\right|
=w^{(p-3)/2}\left|-(p-2)^2 + 2pw - w^2\right|
\end{align*}
is bounded for $w\in(0,p-2)$ and for $p\geq 3$,
Theorem \ref{theorem:non_central_chi_1} gives
\begin{align*}
&\lim_{\nu\to\infty}
\frac{\nu^{(p+1)/4}e^{\nu/2}}{e^{\sql\sqrt{p-2}}} 
 \bigl\{R(\theta,\hat{\theta}_{\JS})-R(\theta,\hat{\theta}_{\JS}^+)\bigr\}\\
&=4c(p-2,0,p)=4\frac{(p-2)^{(p-1)/4}}{\sqrt{2\pi}\exp(p/2-1)}.
\end{align*}
\end{remark}

\end{document}